\documentclass[a4paper, 11pt]{amsart}

\usepackage{amssymb,amsbsy,amsmath,amsfonts,amssymb,amscd}
\usepackage{latexsym}
\usepackage{amsxtra}
\usepackage{amscd}
\usepackage{graphics}

\usepackage{epic}
\usepackage{color}
\input xy
\xyoption{all}

\usepackage{ mathrsfs, amsfonts}

\numberwithin{equation}{section}

\makeindex

\usepackage[all]{xy}
\usepackage{amssymb}

\usepackage{amsmath}

\usepackage{amsthm}

\usepackage{amscd}

\usepackage{amsfonts}

\usepackage{amssymb}

\newcommand{\calO}{\mathcal{O}}

\newcommand{\bbC}{\mathbb{C}}

\newcommand{\bbP}{\mathbb{P}}

\newcommand{\SU}{\textup{SU}}

\newcommand{\Aut}{\textup{Aut}}

\newcommand\Pic{{\text{Pic}}}
\newcommand\Ker{{\text{Ker}}}

\newcommand\om{\omega}

\newcommand{\CC}{\ensuremath{\mathbb{C}}}

\newcommand{\ZZ}{\ensuremath{\mathbb{Z}}}

\newcommand{\hol}{\ensuremath{\mathcal{O}}}

\newcommand\Ga{\Gamma}

\newcommand\ga{\gamma}

\newcommand{\sD}{{\mathcal D}}

\newcommand{\sI}{{\mathcal I}}

\newcommand\Tr{{\textup{Tr}}}

\newtheorem{theorem}{Theorem}[section]
\newtheorem{lemma}[theorem]{Lemma}
\newtheorem{proposition}[theorem]{Proposition}

\newtheorem{corollary}[theorem]{Corollary}
\theoremstyle{definition}     

\newtheorem{example}[theorem]{Example}

\theoremstyle{remark}
\newtheorem{remark}[theorem]{Remark}
\numberwithin{equation}{section}

\newcommand{\beq}{\begin{equation}}
\newcommand{\eeq}{\end{equation}}

\title[Bicanonical map of fake projective planes]
{The Bicanonical map of fake projective planes with an automorphism}

\author[F.Catanese]{Fabrizio Catanese}
\author[J. Keum]{JongHae Keum}
\address {Lehrstuhl Mathematik VIII\\
Mathematisches Institut der Universit\"at Bayreuth\\
NW II,  Universit\"atsstr. 30\\
95447 Bayreuth, Germany}
\email{fabrizio.catanese@uni-bayreuth.de}
\address{School of Mathematics, Korea Institute for Advanced
Study, Hoegiro 85, Dongdaemun-gu, Seoul 02455, Korea}
\email{jhkeum@kias.re.kr}
\thanks{
 The present work was supported by the
 ERC Advanced grant n. 340258, `TADMICAMT' and the National Research Foundation of Korea (NRF). The first  author would also like to
 acknowledge the support and hospitality
 of KIAS, Seoul during his visit as KIAS scholar.}

\date{March 20, 2018} 

\begin{document}

\begin{abstract}
We show, for several fake projective planes with nontrivial automorphism group,  that the bicanonical map is an embedding.
\end{abstract}

\maketitle

\setcounter{section}{0} \section{Introduction}
A smooth compact complex surface with the same Betti numbers as the
complex projective plane ${\bbP}_{\bbC}^2$ is either ${\bbP}_{\bbC}^2$ or is called a fake projective plane. Indeed, such a surface has  $c_2=3$, $c_1^2=K^2=9$ and Picard number 1, thus its canonical class  is either ample or anti-ample, and in the latter case it is
isomorphic to  ${\bbP}_{\bbC}^2$.  In other words a
fake projective plane is a surface of general type
with $p_g=0$ and $c_1^2=3c_2=9$. Furthermore, its universal cover
is the unit 2-ball in $\bbC^2$ by \cite{Aubin} and \cite{Yau} and its fundamental group is a co-compact arithmetic subgroup of PU$(2,1)$ by \cite{Kl}.

Prasad and Yeung   \cite{PY} classified all possible
fundamental groups of fake projective planes. Their proof also shows that the automorphism group of a fake projective plane has order 1, 3, 9, 7, or 21. Then Cartwright and Steger (\cite{CS},  \cite{CS2}) carried out  a computer based group theoretic
enumeration  to obtain a more precise result: there are exactly 50 distinct fundamental groups, each corresponding to a pair of fake projective planes, complex
conjugate but not isomorphic to each other \cite{kk}. They also  computed the automorphism groups of all fake projective planes $X$. There occur four groups:
$$\Aut(X)\cong \{1\},\,\, C_3,\,\, C_3^2\,\, {\rm or}\,\, G_{21}\cong C_7:C_3,$$   where $C_n$ is the cyclic group of order $n$ and $G_{21}$ is the unique non-abelian group of
order $21$ (semidirect product of $C_7$ with $C_3$). Among the 50 pairs,   33 admit a non-trivial group of automorphisms: 3 pairs have $\Aut\cong G_{21}$, 3 pairs have $\Aut\cong C_3^2$ and  27 pairs have $\Aut\cong C_3$.

For each pair of fake projective planes  the first homology group  (abelianization of the fundamental group) $$H_1(X, \mathbb{Z})=Tor(H^2(X, \mathbb{Z}))=Tor(\Pic(X))$$ was also computed in \cite{CS2}.


  By Reider's theorem \cite{reider} (see the next section) the bicanonical system of a ball quotient $X$ is base point free, thus it defines a morphism.
  If the ball quotient $X$ has $\chi(X) \geq 2$, then  $K^2_X \geq 9 \chi(X) \geq 10$, and since a ball quotient cannot
 contain a curve of geometric genus $0$ or $1$, the bicanonical map embeds $X$ unless  $X$  contains a smooth genus $2$ curve $C$ with $C^2=0$, and $C K_X = 2$.


In the case $\chi (X) = 1$, for instance if we have a fake projective plane, we are below the Reider inequality $K^2_X \geq 10$,
and the question of the very-ampleness of the bicanonical system is interesting.

Every fake projective plane $X$ with automorphism group of order 21 cannot contain an effective curve with self-intersection 1, as was first proved in \cite{K13} (published in \cite{K17}, see also \cite{GKMS}). Thus by applying I. Reider's theorem, one sees that the bicanonical map of such a fake projective plane is an embedding into ${\bbP}^9$
(see for instance \cite{dbdc}).

 In addition to these 3 pairs of fake projective planes,  for 7 more pairs we confirm here the very-ampleness of the bicanonical system.

\begin{theorem}\label{Main} For the $7$ pairs of fake projective planes given in Table
\ref{7fpps} the bicanonical map is an embedding into ${\bbP}^9$.
\end{theorem}


\begin{table}[ht]
\label{7fpps}
\renewcommand\arraystretch{1.5}
\caption{Seven pairs of fake projective planes}
\noindent $$
\begin{array}{|c|c|c|l|}
\hline
X& \Aut(X) & H_1(X, \mathbb{Z}) & \pi_1(X/C_3)\\
\hline\hline
(a=15, p=2, \{3,5\}, D_3) & C_3 & C_3\times C_7 & C_3\\
\hline
(a=15, p=2, \{3,5\}, 3_3) &C_3 &C_2^2\times C_3& C_3\\
\hline
(a=15, p=2, \{3,5\}, (D3)_3) &C_3 &C_3& C_3\\
\hline
(\mathcal{C}2, p=2, \{3\}, d_3D_3) & C_3^2 & C_7 & C_7, \{1\}, \{1\}, \{1\}\\
\hline
(\mathcal{C}10, p=2, \{17-\}, D_3) & C_3 & C_7& \{1\}\\
\hline
(\mathcal{C}18, p=3, \emptyset, d_3D_3) &C_3^2 & C_2^2\times C_{13} & C_{13}, Q_8, \{1\}, \{1\}\\
\hline
(\mathcal{C}2, p=2, \emptyset, d_3D_3) &C_3^2 & C_2\times C_{7} & C_{14}, S_3, C_2, C_2\\
\hline
\end{array}
$$
\end{table}

\begin{remark} (1) When $\Aut(X)\cong C_3^2$, there are four quotients corresponding to the four order 3 subgroups of $\Aut(X)$.

(2) By  \cite{CS2}, the fundamental groups of these surfaces  lift to $\SU(2,1)$, thus the tautological line bundle of ${\bbP}^2$ restricted to the ball descends to give a cube root of the canonical bundle of $X$. Note that the first three pairs in Table 1 have 3-torsions, thus have multiple cube roots of $K_X$.
\end{remark}

The first six pairs in Table \ref{7fpps} are covered by the following vanishing result.

\begin{theorem}\label{main} Let $X$ be a fake projective plane with a nontrivial $C_3$-action. Suppose that the quotient surface $X/C_3$ has $H_1(X/C_3, \mathbb{Z})=0$ or $C_3$. Then $$H^0(X, L)=0$$ for any ample line bundle $L$ with $L^2=1$, or equivalently, $X$ contains no effective curve $D$ with $D^2=1$.
\end{theorem}

For the last pair in Table \ref{7fpps} we do not have the vanishing theorem. The surfaces possess  either none or 3 curves $D$ with $D^2=1$.
But even in the latter case we prove the very-ampleness of the bicanonical system (See Theorem \ref{7th}).

In Section 5 we discuss three more pairs with a nontrivial $C_3$-action, for which we prove that the bicanonical map is an embedding outside 3 points, the fixed locus of the $C_3$-action.


\section{Preliminaries}

 For the reader ' s convenience, we recall Reider's theorem \cite{reider} by stating the expanded version given in Theorem 11.4 of \cite{bpv}.

\begin{theorem}\label{reider} \cite{reider} Let L be nef divisor on a smooth projective surface $X$.
\begin{enumerate}
\item Assume that $L^2\ge 5$. If $P$ is a base point of the linear system $|K_X+L|$, then $P$ lies on an effective divisor $D$ such that
\begin{enumerate}
\item $DL=0,\,\,\,\, D^2=-1$, or
\item $DL=1,\,\,\,\, D^2=0$.
\end{enumerate}
\item
Assume that $L^2\ge 9$. If two different points $P$ and $Q$, possibly infinitely near,  are not base points of $|K_X+L|$ and fail to be separated by $|K_X+L|$, then they lie on an effective curve $D$,  depending on $P,\,\,Q$, satisfying one of the following:
\begin{enumerate}
\item $DL=0,\,\,\,\, D^2=-2$ or $-1$;
\item $DL=1,\,\,\,\, D^2=-1$ or $0$;
\item $DL=2,\,\,\,\, D^2=0$;
\item $L^2=9$ and $L$ is numerically equivalent to $3D$.
\end{enumerate}
\end{enumerate}
\end{theorem}

A ball quotient cannot
 contain a curve of geometric genus $0$ or $1$.
By Reider's theorem the bicanonical system of a ball quotient is base point free, thus defines a morphism.  Let $X$ be a fake projective plane and let
$$\Phi_{2,X} : X \to {\bbP}^9$$ be the bicanonical morphism.

\begin{lemma}\label{Fab} Let $X$ be a fake projective plane.
\begin{enumerate}
\item If $D$ is an effective curve on $X$ with $D^2=1$, then $D$ is an irreducible curve of arithmetic genus 3, $h^0(X, \mathcal{O}_X(K_X-D))=0$ and $h^0(X, \mathcal{O}_X(D))=1$. In particular $X$ may contain at most finitely many curves $D$ with $D^2=1$. Their  number is bounded by $|H_1(X, \mathbb{Z})|$.
\item  If two different points $P$ and $Q$ on $X$ (possibly infinitely near) are not separated by $\Phi_{2,X}$, then there is a curve $D$ with $D^2=1$ containing $P, Q$ such that $h^0(D, \mathcal{O}_D((K_X-D)|_D)=1$ and $P+Q$ is the unique effective divisor in the linear system of $(K_X-D)|_D$.  In particular, a curve $D$ with $D^2=1$ may contain at most one pair of points (possibly infinitely near) that are not separated by $\Phi_{2,X}$.  Such a curve $D$ is uniquely determined by $P,Q$.
\item The bicanonical map $\Phi_{2,X}$ yields  an isomorphism with its image of the complement $U$ of a  finite set of points. The bicanonical image $\Sigma$ is a surface with isolated singularities only and $\Phi_{2,X}:X\to \Sigma$  is  the normalization map.
\end{enumerate}
\end{lemma}

\begin{proof} (1) If $h^0(X, D)\ge 2$, then $h^0(X, 4D)\geq 5$. On the other hand, since $4D-K$ is ample, we have $h^1(X, 4D)=h^2(X, 4D)=0$,
 hence, by Riemann-Roch, $h^0(X, 4D)=3$.

(2) By Reider's theorem (Theorem \ref{reider}), if  the bicanonical system $|2K_X|$ does not separate
two points $P, Q$ (possibly infinitely near) then  there exists a divisor $D$ containing both $P,Q$
and such that

 $K \equiv 3 D $ modulo torsion ($\equiv$ denotes here as classical linear equivalence).\\
One sees immediately in fact that, since $ NS(X)$ has rank equal to 1, and its torsion free part is generated by a divisor $L$ with $L^2 =1$, the alternatives (a), (b), (c)
in Theorem \ref{reider} are not possible  ($D^2 \leq 0$ contradicts that $D $ is numerically equivalent to a nontrivial multiple of $L$).
Write then $ K = 3 D + \tau$, and observe that
$$ 2 K = K+ D + ( 2 D + \tau). $$
By \cite{cf} and \cite{4names} in view of the exact sequence
$$ 0  \rightarrow \sI_{P,Q} \om_D (2 D + \tau )  \rightarrow   \om_D (2 D + \tau )  \rightarrow \CC^2  \rightarrow 0$$
it must also hold that
$$ H^1 (D, \sI_{P,Q} \om_D (2 D + \tau )) \cong \CC$$
hence there is an isomorphism $  \sI_{P,Q}  (2 D + \tau ) \cong \hol_D$, thus
$$  \sI_{P,Q}   \cong \hol_D(-2 D - \tau ),$$ $  \sI_{P,Q} $ is invertible and $P+Q$ is the unique divisor of
a section  $\in H^0 ( \hol_D(2 D +  \tau )) \cong \CC.$

 If $P$ and $Q$  are contained in two different curves $D_1$ and $D_2$, then $D_1D_2\ge 2$, which is not possible since the  curves
 $D_1, D_2 $ are numerically equivalent and have self-intersection 1.  This proves the uniqueness of such a curve $D$.

(3) follows from (1) and (2).
\end{proof}

In \cite{K08}, all
possible structures of the quotient surface of a fake projective plane and its minimal
resolution were classified.

\begin{theorem}\label{Keum} \cite{K08} Let $X$ be a fake projective plane with a group $G$ acting on it. Then the fixed locus of any automorphism in $G$ consists of 3 isolated points. Moreover the following hold.
\begin{enumerate}
\item If $G=C_3$, then $X/G$ is a ${\mathbb Q}$-homology
projective plane with $3$ singular points of type
$1/3(1,2)$ and its minimal resolution is a minimal
surface of general type with $p_g=0$ and $K^2=3$.
\item If $G=C_3^2$, then $X/G$ is a ${\mathbb Q}$-homology projective plane
with $4$ singular points of type $1/3(1,2)$ and its
minimal resolution is a minimal surface of general type with
$p_g=0$ and $K^2=1$. \item If $G=C_7$, then $X/G$ is a ${\mathbb
Q}$-homology projective plane with $3$ singular points of type
$1/7(1,5)$ and its minimal resolution is a $(2,3)$-,
$(2,4)$-, or $(3,3)$-elliptic surface.
\item If $G=7:3$, then $X/G$ is a ${\mathbb Q}$-homology projective plane with $4$
singular points, where three of them are of type
$1/3(1,2)$ and one of them is of type
$1/7(1,5)$, and its minimal resolution is a $(2,3)$-,
$(2,4)$-, or $(3,3)$-elliptic surface.
\end{enumerate}
\end{theorem}

Here a ${\mathbb Q}$-homology projective plane is a normal
projective surface with the same Betti numbers as $\bbP_{\bbC}^2$ (cf.
\cite{HK1}, \cite{HK2}).  A normal
projective surface with quotient singularities only is a ${\mathbb Q}$-homology projective plane if its second Betti number is 1  (if the first Betti number were positive, then, since the Picard scheme is compact for a normal surface,
looking at the Albanese map one sees that the Picard number is at least 2).
 A fake projective plane is a nonsingular
${\mathbb Q}$-homology projective plane, hence, by the invariance of the class of the canonical divisor, every quotient is
again a ${\mathbb Q}$-homology projective plane.


 \begin{lemma}\label{geodesic} On a fake projective plane, there is no totally geodesic curve, smooth or singular.
\end{lemma}

\begin{proof}  The proof has the following two steps. We are indebted to  Bruno Klingler and Inkang Kim.

I) In general,  if an arithmetic  ball quotient $X$ contains a totally geodesic ball quotient $Y$ (possibly singular), then $Y$ is arithmetic.

This follows from the definition.
If $\Gamma < G$ is an arithmetic lattice (it means that there exists a number field $k$ such that
$\Gamma$ is commensurable with  $G(\calO_k) < G(k)$ where  $\calO_k$ is the ring of integers of $k$), one may assume that $\Gamma$ is contained in $G(\calO_k)$ up to finite index.
If the arithmetic ball quotient $X$ corresponding to $\Gamma < G$ contains a totally geodesic space $Y$ corresponding to $\Gamma' < G' <G$, then since it is totally geodesic there is an injection from $\Gamma'$ to $\Gamma$
(any loop in $\Gamma'$ cannot be contractible in $\Gamma$.)
Here $G$ is the isometry group of the complex $n$-ball and $G'$ is the isometry group of the complex $m$-ball with $m< n$.
Hence  $\Gamma' < \Gamma < G(\calO_k)$ and consequently
$\Gamma' <  G(\calO_k) \cap G' < G'(\calO_k)$.   Since $G'(\calO_k)$ gives a finite volume ball quotient, so does $\Gamma'$, i.e., $\Gamma'$ is a finite index subgroup of $G'(\calO_k)$.
Hence $\Gamma'$ is an arithmetic lattice in $G'$.

II) In the case where $X$ is obtained from a division algebra (as in the case of a FPP), such a $Y$ has to be a point.


It is known that the defining number field $k$ is totally real, and there is one real place whose real point of $G$ is $\SU(2,1)$, and at all other real places it is $\SU(3)$.
In division algebra terms we have  $(\sD, l)$, where $l$ is a quadratic extension of $k$,  the $k$-group $G$ is identified with $\SU(h)$ where $h$ is a Hermitian form on some power of $\sD$.
Hence $\SU(h)$ is isotropic at one real place of $k$, and anisotropic at all other real places of $k$.
In our case the lattice is contained in the anisotropic places, hence must be contained in a compact group
$U(3) \times U(3) \times ....\times U(3)$. Such a lattice contained in a compact group must be trivial.
In the language of quaternion algebras,
if it is a division algebra (or ramified), the group is contained in a compact group.
\end{proof}

When the central simple algebra $\sD$ splits over $l$ (as in the case of the Cartwright-Steger surface), it is a matrix algebra i.e., contained in a matrix group,
and the ball quotient always contains a totally geodesic curve, possibly singular.

 \begin{lemma}\label{DwithC3} Let $X$ be a fake projective plane. If a curve $D$ on $X$ has $D^2=1$, then it is a smooth curve of genus 3.
\end{lemma}

\begin{proof} For a curve $C$ on a ball quotient $Z$
$$ 3(2g(C')-2)\ge 2K_ZC,$$
where $C'$ is the normalization of $C$, with equality iff $C$ is totally geodesic  \cite{Yau78}.
In our case, since $D^2=1$, we have $K_XD=3$ and $p_a(D)=3$. By Lemma \ref{geodesic} $D$ is not totally geodesic, so the above inequality implies that  $g(D')\ge 3$.

\end{proof}

 \begin{lemma}\label{restTor} Let $C$ be a smooth curve on a smooth complex surface $X$ with $C^2>0$.
Then the natural restriction map
$$Tor\Pic(X)\to \Pic(C)$$
is injective.
\end{lemma}

\begin{proof} Let $\tau$ be a nontrivial torsion line bundle on $X$.  Then it defines an unramified cover $X'\to X$ of finite degree, say $d > 1$. If $\tau|C$ is trivial, then $C$ splits into a disjoint union of curves $C_1,...,C_d$ in $X'$ with $C_i^2=C^2>0$, contradicting the Hodge index theorem.
\end{proof}

\section{Proof of Theorem \ref{main}}

 In this section we prove Theorem \ref{main}.

First, we  state  a general result  on the first homology group of a quotient space $Y = X/G$.

Recall here (cf. 6.7 of \cite{topmethods})  that, for a $\ZZ[G]$-module $M$, the group of coinvariants $M_G$ is the quotient of $M$ by the submodule generated
by $Im (g-1)$, for $g \in G$. In particular,  $M_G$ is the quotient of $M$ modulo the relations $g_i (x) \equiv x$, for a system of generators $g_i$ of $G$.

The functor $M \mapsto M_G$ is the same as tensor product with the trivial $\ZZ[G]$-module $\ZZ$, i.e.,  $M_G = M \otimes_{\ZZ[G]} \ZZ$.
 Recall that tensor product is right exact,
and that  the left derived
functors are the homology groups $H_i(G, M)$.
In particular, $$H_1 (G, \ZZ) = G^{ab}, \ H_i(G, \ZZ[G] )= 0, \ \forall i \geq 1, H_0(G, \ZZ[G] ) = \ZZ.$$

\begin{proposition}\label{G-coinv}
Assume that $X$ is a good topological space (arcwise connected and semilocally 1-connected) and assume that the group $G$ is
 a properly  discontinuous group of homeomorphisms of $X$. Let $Y = X/G$ be the quotient space. Then:

 (I) If $G$ is generated by the stabilizer subgroups  $G_x$, then $H_1(X/G, \ZZ)$  is a quotient of the group of co-invariants $H_1(X, \ZZ)_G$,
$$ H_1(X, \ZZ)_{G} \twoheadrightarrow  H_1(X/G, \ZZ).$$

(II) More generally, if $K(X)$ is the normal subgroup generated by the stabilizer subgroups  $G_x$, then $H_1(X/G, \ZZ)$  is an extension of a  quotient of the group of co-invariants  $H_1(X, \ZZ)_G$ by the abelianization of $G/K(X)$,
 i.e. the following sequence is exact:
$$  H_1(X, \ZZ)_{G}  \rightarrow H_1(X/G, \ZZ)   \rightarrow (G/K(X))^{ab} \rightarrow 0 .$$

 (III) If $X$ is homotopically equivalent to a simplicial complex on which $G$ acts simplicially, and with only isolated  fixed points, then:

 the kernel of the homomorphism $  H_1(X, \ZZ)_{G}  \rightarrow H_1(X/G, \ZZ) $ is generated by the image of a group $H_1 (G, Z_0)$  sitting in an exact sequence:

 $$ H_2 (G, \ZZ) \rightarrow H_1 (G, Z_0) \rightarrow H_1 (G, C_0) \rightarrow  H_1 (G, \ZZ ) = G^{ab} ,$$

 where $H_1 (G, C_0)$ is the
  direct sum of groups of the form $$H_1(G, \ZZ[G/G']) \cong (G' ) ^{ab} ,$$   where $G'$ is a subgroup of $G$ (here $\ZZ[G/G']$ is just a module over the group ring $\ZZ[G]$).

 (IV) In particular, if $G$ is a finite abelian group, $G$ is generated by stabilizers, and $H_1(X, \ZZ)_G$ is a torsion  group of order relatively prime to $|G|$, then $H_1(X, \ZZ)_G = H_1 (X/G, \ZZ)$.
\end{proposition}

\begin{proof}
Let $p : \tilde{X} \rightarrow X$ be the universal cover, $\pi : = \pi_1(X)$, so that $ X = \tilde{X} / \pi$.

The group $G$ (cf. 6.1 of \cite{topmethods}) admits an exact sequence
$$ 1 \rightarrow \pi \rightarrow \Ga \rightarrow  G \rightarrow 1,$$

where $\Ga$ acts propery discontinuosly on $ \tilde{X} $ and  $ Y = X / G  = \tilde{X} / \Ga$.

By the theorem of Armstrong \cite{arm1, arm2} we have that $ \pi_1(Y)  = \Ga / K$,
where $K$ is the subgroup generated by stabilizers $\Ga_z$, for $ z \in   \tilde{X} $. As $\pi$ acts freely, $\Ga_z$ maps isomorphically to the stabilizer $G_x$, if $ x = p(z)$.
Indeed, for each $ z \in p^{-1}(x)$ there is a splitting of $G_x$, and changing $z$ only changes $\Ga_z$ up to conjugation by $\pi$.

In particular, $K$ maps onto the normal subgroup $K(X)$ of $G$ generated by the stabilizers $G_x$, and we have an exact sequence
$$ 1 \rightarrow \pi / ( \pi \cap K ) \cong ( \pi K )/ K\rightarrow \pi_1 (Y) = \Ga/ K  \rightarrow  G / K(X) = \Ga / (\pi K) \rightarrow 1.$$
Set  $$H : = H_1(X, \ZZ)= \pi^{ab}, \  H' : = H_1(Y, \ZZ) = \pi_1(Y)^{ab} = ( \Ga/ K)^{ab} =  \Ga / (K [\Ga, \Ga]),$$
hence an exact sequence
$$ 1 \rightarrow  (\pi K [\Ga, \Ga]) / (K [\Ga, \Ga])   \rightarrow H'  \rightarrow  (G / K(X))^{ab} = \Ga / (\pi K [\Ga, \Ga]) \rightarrow 1.$$
The left hand side equals $ \pi / ( \pi \cap  (K  [\Ga, \Ga]))$, and is clearly a quotient of $H = \pi / [\pi, \pi]$. Moreover, for each $\ga \in \Ga, \phi \in \pi$,
we have that $\ga \phi \ga^{-1} $ and $ \phi$ have the same image in $H'$. So, for each $g \in G$, $G$ acts trivially on the image of $H$ inside the kernel of the surjection
$ H'   \rightarrow  (G / K(X))^{ab}$. So, we get an exact sequence
$$ H_G   \rightarrow H'   \rightarrow  (G / K(X))^{ab} \rightarrow 0,$$
and (I) and (II) are proven.

For (III), observe that H is the first homology group of the complex of simplicial chains in $X$
$$ C_2 \rightarrow C_1 \rightarrow Z_0,$$
where we take as $Z_0$  the group of degree zero $0$-chains.

We hence have several exact sequences, where $Z_i$ is the subgroup of $i$-cycles, $B_i$ is the group of $i$-boundaries:
$$ 0   \rightarrow Z_1 \rightarrow C_1  \rightarrow Z_0 \rightarrow 0,$$
$$ 0   \rightarrow B_1 \rightarrow Z_1  \rightarrow H \rightarrow 0,$$
$$ 0   \rightarrow Z_2 \rightarrow C_2  \rightarrow B_1 \rightarrow 0.$$

Applying the functor of coinvariants we get exact sequences:
$$ B_{1,G} \rightarrow Z_{1,G}  \rightarrow H_G \rightarrow 0,$$
$$  C_{2,G}   \rightarrow B_{1,G} \rightarrow 0,$$
$$  H_1 (G, Z_0)    \rightarrow Z_{1,G} \rightarrow C_{1,G}  \rightarrow Z_{0,G} \rightarrow 0.$$

Denote  now by $H''$  the homology of the complex
$$ C_{2,G} \rightarrow C_{1,G}  \rightarrow Z_{0,G}.$$
By what we have observed above, $H''$  is a quotient of $Z_{1,G} $ by the subgroup generated by the image of $B_{1,G}$ and by the image of $H_1 (G, Z_0) $,
hence a quotient of $H_G$ by the image of $H_1 (G, Z_0) $.

Now, by our hypothesis, for $ i \geq 1$,  $C_{i,G} = C'_i$, the group of simplicial $i$-chains on $ Y = X/G$.
This is true since $G$ acts freely on $i$-chains, for $i \geq 1$.

And the first homology group $H'$ of $X/G$ is the homology of the complex
$$ C'_2 \rightarrow C'_1 \rightarrow Z'_0.$$

Moreover, $Z_0 \rightarrow  Z'_0$
factors through $Z_{0,G}$, hence $H' / H''$ equals   the quotient of two  kernels
$$ ker (C_{1,G}  \rightarrow  Z'_0)   / ker ( C_{1,G}\rightarrow  Z_{0,G} ),$$
and we shall now again see that it is isomorphic to $(G/K(X))^{ab}$.
Indeed, if a 1-cycle on $Y$ maps to zero in $Z'_0$, then it lifts to a 1-cycle on $X$ with boundary of the form $x - g(x)$.
Adding  zero,  a path from one vertex to another minus the same path, we can take $x$ to lie on any fibre over a vertex of $Y$.
In particular, if $g \in G_x$, we get $x - g(x) = 0$, hence similarly, adding a path from $x$ to $z$, minus its transform via $g$, we get  $z - g(z) = x - g(x)= 0$ for each other vertex $z$. Finally,
since $ z = g(z) \Rightarrow  h z = g h (z) = h g (z)$ (use the fact that we work in the group of coinvariants for the second equality), we obtain that our quotient $H' / H''$
equals $(G/K(X))^{ab}$.

Finally, the module $C_0$ is a direct sum $M \oplus M'$, where $M$ is a  free module  corresponding to vertices of $X$ on which $G$ acts freely, and $M'$ is the direct sum of  modules corresponding to orbits of vertices
with a nontrivial stabilizer $G'$, hence
modules of the form $\ZZ[G/G']$.

	For the former summand $H_1 (G, M) = 0$, for the latter  the Hochshild-Lyndon-Serre spectral sequence, if $G'$ is a normal subgroup, yields $$H_1 (G, \ZZ[G/G']) = (G')^{ab}.$$
	If $G'$ is not normal, the same assertion follows from Shapiro's lemma (see \cite{brown}, 6.2,
	page 73), since the $\ZZ[G]$-module $\ZZ[G/G']$  is just the  representation of $G$ induced  from the trivial representation $\ZZ$ of $G'$.
	From the exact sequence
	$$ 0 \rightarrow Z_0 \rightarrow C_0 \rightarrow \ZZ \rightarrow 0 $$
	the exact group  homology  sequence yields an exact sequence
	$$ H_2 (G, \ZZ) \rightarrow H_1 (G, Z_0) \rightarrow H_1 (G, C_0) \rightarrow  H_1 (G, \ZZ ) = G^{ab} \rightarrow Z_{0,G} \rightarrow C_{0,G} \rightarrow \ZZ$$

(IV) follows easily once we observe  that the finite group $ H_2 (G, \ZZ) \cong  H^2 (G, \CC^*)$ has exponent dividing $|G|$ (cf. theorem 6.14 of \cite{jac2}.
\end{proof}

\begin{example}
Let $X$ be a hyperelliptic curve of genus $g$,  and $G = \ZZ/2$ acting via the hyperelliptic involution, so that $X/G = \bbP^1$. Here $H_G = (\ZZ/2)^{2g}$, while  $H_1 (X/G, \ZZ) = H_1 (\bbP^1, \ZZ) = 0$.

More generally, take $X$ be the product of two hyperelliptic curves of genera $g_1, g_2$, and $G = \ZZ/2$ acting diagonally  via the two  hyperelliptic involutions. Again the quotient is simply connected,
the locus of  fixed points has complex codimension 2,  $H_G = (\ZZ/2)^{2g_1 + 2 g_2}$.

These examples show the optimality of the results of  Proposition \ref{G-coinv}.
\end{example}

 If $G=\langle g\rangle$ is cyclic, then $g^i-1=(g-1)(g^{i-1}+\ldots +1)$, hence
$Im(g^i-1)\subset Im(g-1)$ for all $i$ and,  as already mentioned,
$$M_G=M_g:=M/Im(g-1).$$

\begin{lemma}\label{coinv=inv} (I) If a   cyclic group $G$ of order  $m$ acts on an abelian group $H$, and if $m$ is coprime to the order $|h|$ of every element  $h \in H$, then $$H_G\cong H^G.$$

(II) If a finite cyclic group $G$ acts on $X$  with only isolated fixed points, $G$ is generated by the stabilizer subgroups and  $H_1(X, \ZZ)$ is finite and has order coprime to $|G|$, then
$$H_1 (X/G, \ZZ)=H_1(X, \ZZ)_G = H_1(X, \ZZ)^G.$$
\end{lemma}

\begin{proof}
(I) Let $g\in G$ be a generator: $g$ has order  $m$.  Consider the trace homomorphism $\Tr(g)=1+g+g^2+...+g^{m-1}:H\to H$. Since $$(g-1)\Tr(g)=\Tr(g)(g-1)=g^m-1=0$$
we have
$$Im(\Tr(g))\subset \Ker(g-1)=H^G,\,\,\,Im(g-1)\subset \Ker(\Tr(g)).$$
We will show that both are equalities under the assumption.

If  $h\in \Ker(g-1)=H^G$, then, choosing an integer $a$ such that $am\equiv 1$ mod $|h|$, we see that $\Tr(g)(ah)=mah=h$, hence $h\in Im(\Tr(g))$.

If $h\in \Ker(\Tr(g))$, then, again by choosing a positive integer $a$ such that $am\equiv 1$ mod $|h|$, one has
$$h=-g(h)-g^2(h)-...-g^{am-1}(h)$$
$$=(g-1)(g(h)+2g^2(h)+...+(am-1)g^{am-1}(h))-(am-1)h$$
$$=(g-1)(g(h)+2g^2(h)+...+(am-1)g^{am-1}(h))\in Im(g-1).$$

(II) This follows from (I) and Proposition  \ref{G-coinv}.
\end{proof}

\begin{corollary}\label{FPP21} Let $X$ be a fake projective plane with Aut$(X)\cong C_7:C_3$. Then
   $$H_1(X, \mathbb{Z})^{C_7}\cong H_1(X, \mathbb{Z})_{C_7}\cong H_1(X/C_7, \mathbb{Z})=\pi_1(X/C_7)\cong 0\,\,{\rm or}\,\,C_2.$$  More precisely
    the $C_7$ action on  $H_1(X, \mathbb{Z})$ fixes  no 2-torsion element in the case of  $H_1(X, \mathbb{Z})\cong C_2^3,\,\, C_2^6,$ and only one in the case of  $H_1(X, \mathbb{Z})\cong C_2^4$.
\end{corollary}

\begin{proof}  Recall that, by  \cite{CS2},
 the three pairs of fake projective planes with Aut$(X)\cong G_{21}$ have torsion groups
$$ H_1(X, \mathbb{Z})\cong C_2^3,\,\, C_2^4,\,\, C_2^6,$$  respectively. By Proposition  \ref{G-coinv} and Lemma \ref{coinv=inv}
$$H_1(X, \mathbb{Z})^{C_7}\cong H_1(X, \mathbb{Z})_{C_7}\cong H_1(X/C_7, \mathbb{Z}).$$
   By Theorem \ref{Keum}, since an $(a,b)$-elliptic surface has fundamental group isomorphic to the cyclic group of order $gcd(a,b)$  \cite{D}, we see that $$\pi_1(X/C_7)=\pi_1(X/C_7)^{ab}=H_1(X/C_7, \mathbb{Z})$$ is of order at most 3, hence either $0$ or $C_2$ (this coincides with the computation of $\pi_1(X/C_7)$ in \cite{CS2}.)  Since the polynomial $x^7 -1 $ in $C_2 [x]$ is the product of thee prime factors
   $ (x+1 ) (x^3 + x^2 + 1) (x^3 + x + 1)$ we see that any linear action of  $C_7$ on a vector space $C_2^n$ is a direct sum of  subspaces of cardinality $2$ or $8$.
    Thus the $C_7$ action on  $H_1(X, \mathbb{Z})$ fixes  no 2-torsion element in the case of  $H_1(X, \mathbb{Z})\cong C_2^3,\,\, C_2^6,$ and one in the case of  $H_1(X, \mathbb{Z})\cong C_2^4$.
\end{proof}

 This is no strange. In fact, $Aut(C_2^3)\cong GL(3,2)\cong PSL(2,7)$, a simple group of order 168 containing a subgroup $\cong G_{21}$.

 \begin{remark} The paper by S.-K. Yeung [A surface of maximal canonical degree, Math. Ann. 368 (2017), 1171-1189] contains wrong proofs, thus the main result is not proven at all. His proof of the base point freeness of $|K_M|$ relies on his wrong claim that
the $C_7$ action on  $H_1(X, \mathbb{Z})=C_2^4$ is trivial, which is not true by our Corollary \ref{FPP21}.
\end{remark}

\begin{lemma}\label{little} Let $F$ be a finite abelian group  of order not divisible by 9. Suppose that $F$ admits an order 3 automorphism $\sigma$ such that
the group of co-invariants  $F_{\sigma}\cong C_3$ or $0$. Then for every $t\in F$
$$t+\sigma(t)+\sigma^{2}(t)=0\,\,{\rm  in}\,\,F.$$
\end{lemma}

\begin{proof}  In the case $$ 0 =  F_{\sigma} = F / (\sigma -1)F$$ $1-\sigma$ is invertible hence
$$ 0 = \sigma^3 - 1 =  (\sigma -1) (\sigma^2 + \sigma +1) \Rightarrow (\sigma^2 + \sigma +1)= 0.$$

Note that the action of $\sigma$ on $F$ is the product of its actions on the $p$-primary summands of $F$.

In the second case $F_{\sigma}\cong C_3$, the previous argument applies for  $p$-primary summands of $F$ for $p\neq 3$.

For $p=3$ clearly $\sigma = 1$, hence  $(\sigma^2 + \sigma +1) = 3 = 0$.
\end{proof}

Now we prove Theorem \ref{main}.\\
Suppose that $X$ has no 3-torsion. Then $K_X$ has a unique cube root $L_0$. Then $L_0$ is fixed by every automorphism.
Let $L$ be an ample line bundle with $L^2=1$. Then $L=L_0+t$ for some torsion line bundle $t$, and
$$\sigma^*(L_0+t)=L_0+\sigma^*(t),\,\,\,\sigma^{2*}(L_0+t)=L_0+\sigma^{2*}(t).$$ By the above lemma, $t+\sigma^*(t)+\sigma^{2*}(t)=0$. Thus
$$(L_0+t)+\sigma^*(L_0+t)+\sigma^{2*}(L_0+t)=3L_0+(t+\sigma^*(t)+\sigma^{2*}(t))=3L_0=K_X.$$
If $L_0+t$ is effective, then all the three summands are effective, so is $K_X$, a contradiction. Thus $H^0(X, L_0+t)=0$.

\medskip
Suppose that $X$ has  nontrivial 3-torsion. In this case, $K_X$ has a cube root $L_0$, not unique (see Remark under Table 1).
Note that $L_0$ may not be fixed by the order 3 automorphism $\sigma$.\\
If $\sigma^*(L_0)=L_0$, then the previous argument shows that $H^0(X, L_0+t)=0$.\\
If $\sigma^*(L_0)=L_0+t_3$ for some 3-torsion $t_3$, then for any torsion line bundle $t$
$$\sigma^*(L_0+t)=L_0+t_3+\sigma^*(t),\,\,\,\sigma^{2*}(L_0+t)=L_0+2t_3+\sigma^{2*}(t).$$ By the above lemma, $t+\sigma^*(t)+\sigma^{2*}(t)=0$.
Thus
$$(L_0+t)+\sigma^*(L_0+t)+\sigma^{2*}(L_0+t)=3L_0+3t_3=K_X.$$
Since $K_X$ is not effective, none of the three summands is effective.

\section{Fake projective planes with $\Aut(X)=C_3^2$ and $H_1(X, \mathbb{Z})=C_{14}$}

This is one of the 3 pairs of fake projective planes with  $\Aut(X)=C_3^2$. The other two pairs are also listed in Table 1.
Note that the unique 2-torsion element  is fixed by every automorphism.

 If $\Aut(X)=C_3^2$ acts trivially on $H_1(X, \mathbb{Z})$, then $H_1(X, \mathbb{Z})_{\Aut(X)}=C_{14}$.
This group has order coprime to $3$, thus by Proposition \ref{G-coinv}
$H_1(X/\Aut(X), \mathbb{Z})=C_{14}$.
But $X/\Aut(X)$ has four $A_2$-singularities and its minimal resolution is a numerical Godeaux surface (a minimal surface of general type with
$p_g=0$ and $K^2=1$). This is impossible, as a numerical Godeaux surface
has a  torsion group of order   $\le 5$.  Thus, $\Aut(X)=C_3^2$ does not act trivially on $H_1(X, \mathbb{Z})$.

Note that $$ Aut(H_1(X, \mathbb{Z}))\cong (C_2\times C_7)^* \cong C_{6}.$$  Thus  we have
$$\Ker(Aut(X)\to Aut(H_1(X, \mathbb{Z}))\cong C_3.$$
If $\sigma\in \Aut(X)=C_3^2$ acts trivially on $H_1(X, \mathbb{Z})$, then by  Proposition \ref{G-coinv} and Lemma \ref{coinv=inv}
$$H_1(X/\langle\sigma\rangle, \ZZ)=H_1(X, \ZZ)_{\sigma}=H_1(X, \ZZ)^{\sigma}=C_{14}.$$
If $\sigma\in \Aut(X)=C_3^2$ does not act trivially on $H_1(X, \mathbb{Z})$, then it fixes the 2-torsion element, permutes the six 7-torsion  elements and the six 14-torsion  elements, hence
$$H_1(X/\langle\sigma\rangle, \ZZ)=H_1(X, \ZZ)_{\sigma}=H_1(X, \ZZ)^{\sigma}=C_{2}.$$
This coincides with the computation of Cartwright and Steger \cite{CS2}:
$$\pi_1(X/C_3)=C_{14},\,\, S_3,\,\,C_2,\,\, C_2$$
respectively for four order 3 subgroups of $\Aut(X)=C_3^2$ (see Table 1).

Since $X$ has no 3-torsion, it has a unique cubic root of $K_X$. Let $L_0\in\Pic(X)$ be the unique cubic root of $K_X$.

First we recall the following vanishing result from (\cite{K17}, Theorem 0.2 and its proof).

\begin{theorem}\label{K2017}\cite{K17} Let $X$ be a fake projective plane with $\Aut(X)\cong C_3^2$.
Then $H^0(X, 2L_0+t)=0$ for any $Aut(X)$-invariant torsion line bundle $t$. In particular, $H^0(X, 2L_0)=0$
\end{theorem}

\begin{remark}  Among the three pairs of fake projective planes with  $\Aut(X)\cong C_3^2$,  only the pair with $H_1(X, \mathbb{Z})=C_{14}$ has an $Aut(X)$-invariant non-trivial torsion line bundle, which corresponds to the unique 2-torsion in $H_1(X, \mathbb{Z})$.
It follows that for this pair of fake projective planes, $\mathcal{O}_X, -(L_0+t_2), -2L_0$ also form an exceptional collection.
\end{remark}

\begin{lemma}\label{} For a fake projective plane $X$ with $\Aut(X)=C_3^2$ and $H_1(X, \mathbb{Z})=C_{14}$,
$$H^0(X, L_0+t)=0$$ for any torsion element $t\in \Pic(X)$, except possibly for three 14-torsion elements that are rotated by an order 3 automorphism.
\end{lemma}

\begin{proof}
Suppose that $t=t_7$ is a 7-torsion element.  We know that there is an automorphism $\sigma \in Aut(X)$ such that $\sigma^*(t)=2t$ (by replacing it by $\sigma^2$ if $\sigma^*(t)=4t$).
Thus $$\sigma^*(L_0+t)=L_0+2t, \,\,\,\sigma^{2*}(L_0+t)=L_0+4t.$$
Since  $p_g(X)=H^0(X, K_X)=0$ and
$$(L_0+t)+(L_0+2t)+(L_0+4t)=3L_0=K_X,$$
we have $H^0(X, L_0+t)= 0$.

\medskip
Supppose that $t=t_2$ is the unique 2-torsion  element. It is fixed by every automorphism. Thus
$$(L_0+t_2)+\sigma^*(L_0+t_2)+\sigma^{2*}(L_0+t_2)=3(L_0+t_2)=K_X+t_2$$ and one cannot use the previous argument.
 But the vanishing $H^0(X, L_0+t_2)= 0$ follows from Theorem \ref{K2017}, since $2(L_0+t_2)=2L_0$ is $\Aut(X)$-invariant.

\medskip
Suppose that $H^0(X, L_0+t_2+t_7)\neq 0$ for some 7-torsion element $t_7$.
Theorem \ref{K2017} we know that there is an automorphism $\sigma \in Aut(X)$ such that $\sigma^*(t_7)=2t_7$.
Thus $$\sigma^*(L_0+t_2+t_7)=L_0+t_2+2t_7, \,\,\,\sigma^{2*}(L_0+t_2+t_7)=L_0+t_2+4t_7$$
and these two line bundles are effective.
We know that  $H^0(X, 2L_0)= 0$. Since $$(L_0+t_2+at_7)+(L_0+t_2+(7-a)t_7)=2L_0,$$ we have
$$H^0(X, L_0+t_2+6t_7)=H^0(X, L_0+t_2+5t_7)=H^0(X, L_0+t_2+3t_7)=0.$$
\end{proof}

From now on, assume that $H^0(X, L_0+t_2+t_7)\neq 0$ for some 7-torsion element $t_7$. Then $$H^0(X, L_0+t_2+t_7)\cong\mathbb{C}.$$
Let $D_1$ be the unique effective curve in the linear system, i.e.
$$D_1 \equiv L_0+t_2+t_7.$$ Define
$$D_2=\sigma^*D_1\equiv L_0+t_2+2t_7, \,\,\,D_3=\sigma^{*}D_2\equiv L_0+t_2+4t_7.$$


 There is another automorphism $\nu\in Aut(X)$ acting trivially on $H_1(X, \mathbb{Z})=C_{14}$. Then
$$\nu^*(M)=M$$ for any line bundle $M$. In particular, $$\nu^{*}(D_i)=D_i,\,\,\,i=1,2,3.$$
By Lemma \ref{DwithC3},  each $D_i$ is a smooth curve of genus 3.
Note that the intersection number $$D_iD_j=1,\,\,\,i,j=1,2,3.$$ Hence  $D_i$ and $D_j$
intersect transversally in a point $x_{ij}$.
 Then $D_1, D_2, D_3$ form a triangle with vertices $x_{ij}$. (If the three curves intersect at a point $x$, then  both $\sigma$ and $\nu$ fix $x$, impossible by Theorem \ref{Keum}.)  We know that the fixed locus of $\nu$ consists of three isolated points, so we infer that
$$Fix(\nu)=\{x_{12}, x_{23}, x_{31}\}.$$

\begin{theorem}\label{7th} Let $X$ be a FPP with $\Aut(X)=C_3^2$ and $H_1(X, \mathbb{Z})=C_{14}$. The bicanonical map $\Phi_{2,X}$ of $X$ is an embedding.
\end{theorem}

\begin{proof} If two  different points $P$ and $Q$, with $Q$ possibly infinitely near to $P$,  are not separated by the bicanonical system, then they must belong to one of  the  three curves $D_i$, say $D_1$. We know that
$P+Q$ is the unique divisor of $H^0(D_1, \mathcal{O}_{D_1}((K_X-D_{1})|D_1))$. Since $\nu$ preserves the line bundle $K_X-D_1$ and the curve $D_1$, it preserves the divisor $P+Q$. Since $\nu |D_1$ is of order 3, both $P$ and $Q$ are fixed points of $\nu |D_1$.
 If $Q$ is not  infinitely near to $P$, then  $$P+Q=x_{12}+x_{13},$$ thus
in our previous notation, we should have
$$ \hol_{D_1} (2 L_0 + 6t_7)\cong \hol_{D_1} ( D_2 + D_3) \cong \hol_{D_1} (x_{12}+x_{13})\cong \hol_{D_1}(K_X-D_1) $$
$$ \cong  \hol_{D_1} (2 L_0 + t_2 - t_7),$$ hence
$$\hol_{D_1} (t_2 ) \cong \hol_{D_1} .$$
This contradicts however Lemma \ref{restTor}.

 If $Q$ is   infinitely near to $P$, then $P+Q | D_1 =2x_{12}$ or $P+Q | D_1 =2x_{13}$. In the former case  we must have
 $$ \hol_{D_1} (2 L_0 + 4 t_7)\cong \hol_{D_1} ( 2 D_2 ) \cong \hol_{D_1} (2 x_{12} )\cong \hol_{D_1}(K_X-D_1) $$
$$ \cong  \hol_{D_1} (2 L_0 + t_2 - t_7),$$ hence
$$\hol_{D_1} (t_2 + 5 t_7) \cong \hol_{D_1} ,$$
contradicting again Lemma \ref{restTor}. The argument in the latter case is identical.
\end{proof}

 \section{Fake projective planes with $H_1(X, \mathbb{Z})=C_6$}

Among the 50 pairs of fake projective planes, exactly three pairs have $H_1(X, \mathbb{Z})=C_6$, as listed in Table 2.
Moreover they all have $\Aut(X)=C_3$. Since $\Aut(H_1)\cong C_2$ in this case, the automorphism group acts trivially on $H_1(X, \mathbb{Z})=C_6$.
For each of the 6 surfaces the previous argument shows that  $\Aut(X)$ preserves every line bundle $L$ with $L^2=1$, and there are on $X$ at most two curves with self intersection 1.
 If $X$ has only one curve  $D $ with $D^2=1$, then $\Aut(X)=C_3$ fixes 2 points on $D$ and the bicanonical map embedds away from the two points.  If $X$ has two curves  $D_1$, $D_2 $ with $D_1^2=D_2^2=D_1D_2=1$, then $\Aut(X)=C_3$ fixes three points on $D_1\cup D_2$ and the bicanonical map embeds away from the three points.

\begin{table}[ht]
\label{3fpps}
\renewcommand\arraystretch{1.5}
\caption{Fake projective planes with $H_1(X, \mathbb{Z})=C_6$}
\noindent $$
\begin{array}{|c|c|c|l|}
\hline
X& \Aut(X) & H_1(X, \mathbb{Z}) & \pi_1(X/C_3)\\
\hline\hline
(a=15, p=2, \{3\}, (D3)_3) & C_3 & C_6 & C_6\\
\hline
(\mathcal{C}18, p=3, \{2\}, (dD)_3) &C_3 & C_6 & C_6\\
\hline
(\mathcal{C}18, p=3, \{2\}, (d^2D)_3) &C_3 & C_6 & C_6\\
\hline
\end{array}
$$
\end{table}

\section{Remarks on Yeung's Paper}

In this section we point out that, due to many wrong proofs, the  main result of the paper by S.-K. Yeung [Very ampleness of the bicanonical line bundle on compact complex 2-ball quotients, Forum Math. 2017;30(2): 419-432]
is not proven at all.
On page 426 his computation $$\hat{C}=\tau^* C-E_{11}/2-E_{21}/2$$ is false (see the next sentence for a correct calculation), thus his proof of Case (b) does not stand. Consequently his proof of Case (c) and (d) does not stand either, since it is based on Case (b). His arguments also rely on another wrong claim of him that a curve $C$ with $C^2=1$, if exists, is unique. No uniqueness is guaranteed in general, as is explained in our Lemma \ref{Fab}.
 For a  Galois or non-Galois cover $p$, it may happen that $p^*p_*C$ contains a component different from $C$.  Finally his proof of Case (a), the Cartwright-Steger surface case, is not complete at all, based on too strong assumptions which should be proved. A correct proof for the Cartwright-Steger surface case was given in \cite{K18}.

\bigskip
Let $X$ be a fake projective plane with a nontrivial $C_3$-action. Suppose $X$ contains a $C_3$-invariant curve $D$ with $D^2=1$. Here is a correct calculation of the image of $D$ on the minimal resolution $\tau:Y\to X/C_3$.

Since the fixed locus of $C_3$ consists of 3 isolated points \cite{K08}, applying the Hurwitz formula to the $C_3$-action on $D$ it is easy to see that it fixes 2 points on $D$, say $x_1$ and $x_2$, as explained in the previous section, and that the quotient curve $D/C_3$ has genus $1$. The image $\bar{D}\subset X/C_3$ has self-intersection
$$\bar{D}^2=D^2/3=1/3$$
and the proper transform $D'\subset Y$ intersects exactly one of the two $(-2)$-curves $E_{i1}, E_{i2}$ that lie over the $A_2$-singular point $\bar{x_i}\in X/C_3$. One may assume $D'E_{i1}=1, D'E_{i2}=0$ for $i=1, 2$. Then a simple linear algebra calculation determines the coefficients  in the following numerical equivalence:
\beq\label{D'} D'=\tau^*\bar{D}-(2E_{11}+E_{12})/3-(2E_{21}+E_{22})/3. \eeq
Thus
\beq\label{D'^2} D'^2=\bar{D}^2-2/3-2/3=-1, \eeq which does not contradict $g(D')=g(D/C_3)=1$, because $K_Y=\tau^*K_{X/C_3}$ and $D'K_Y=(\tau^*\bar{D})(\tau^*K_{X/C_3})=\bar{D}K_{X/C_3}=(1/3)DK_X=1$.


\bigskip

{\bf Acknowledgements:} We would like to thank  Michael L\"onne for a useful suggestion concerning the proof of (III) of Proposition \ref{G-coinv}, and  Inkang Kim and Bruno Klingler for
useful email correspondence concerning the non existence of geodesic curves on FPP's.


\end{document}